\documentclass[11pt, letter]{article}

\usepackage[english]{babel}
\usepackage[utf8]{inputenc}
\usepackage[T1]{fontenc}
\usepackage{stmaryrd}
\usepackage{footmisc}
\usepackage[colorlinks=true, allcolors=blue]{hyperref}
\usepackage{algpseudocode}
\usepackage{algorithm}
\usepackage{enumitem}
\hypersetup{%
  colorlinks=true,
  linkcolor=blue,
  linkbordercolor={0 0 1}
}

\usepackage{tikz}
\usetikzlibrary{shapes.geometric, arrows.meta, positioning}

\tikzset{
  process/.style = {
    rectangle, rounded corners, minimum width=3.5cm, minimum height=1.2cm,
    text centered, draw=black, font=\normalsize, text width=7cm, align=center
  },
  arrow/.style = {
    thick,->,>=stealth
  }
}

\usepackage{fullpage}
\usepackage{comment}
\usepackage{mathrsfs}
\usepackage{amsmath}
\usepackage{amssymb}
\usepackage{bbm}
\usepackage{graphicx}
\usepackage{stackengine}
\usepackage{amsthm}
\usepackage{subcaption}

\usepackage{thmtools}
\usepackage{thm-restate}
\usepackage{hyperref}
\usepackage{cleveref}
\usepackage{natbib}
\usepackage{authblk}
\usepackage[colorinlistoftodos]{todonotes}

\usepackage{amsfonts}
\newcommand\restr[2]{{% we make the whole thing an ordinary symbol
  \bigg.\kern-\nulldelimiterspace % automatically resize the bar with \bigg
  #1 % the function
  \vphantom{\big|} % pretend it's a little taller at normal size
  \bigg|_{#2} % this is the delimiter
  }}

\newtheorem{theorem}{Theorem}[section]
\newtheorem{lemma}[theorem]{Lemma}

\newtheorem{corollary}[theorem]{Corollary}

\newtheorem{definition}[theorem]{Definition}

\newtheorem{assumption}{Assumption}
\theoremstyle{remark}

\newtheorem{fact}{Fact}

\newcounter{note}[section]
\newcommand{\anote}[1]{\refstepcounter{note}{\bf\color{teal} April:}
  {\sf \color{teal}  #1}}

\newcommand{\norm}[1]{\left|\left| #1 \right|\right|}

\newcommand{\By}[1]{\mathbf{y}}

\DeclareMathOperator*{\argmin}{arg\,min}
\DeclareMathOperator*{\diag}{diag}

\usepackage{bibentry}

\usetikzlibrary{shapes.geometric}

\title{Finding a Multiple Follower Stackelberg Equilibrium: \\A Fully First-Order Method}

\author[1]{April Niu}
\author[2]{Kai Wang}
\author[1]{Juba Ziani}
\affil[1]{School of Industrial and Systems Engineering, Georgia Tech}
\affil[2]{School of Computational Science and Engineering, Georgia Tech}

\begin{document}

\maketitle

\begin{abstract}
This paper studies Stackelberg games with multiple followers and continuous strategy spaces, where a single leader first commits to a strategy, then $k$ followers ($k>1$) play a simultaneous game in response to the leader's decision. 
We study the complexity of finding \emph{$\epsilon$-stationary Stackelberg equilibria}, where neither the leader nor the followers want to deviate from the current strategies with gradient norm greater than $\epsilon$. 
In this work, we propose the first fully first‑order method to compute a $\epsilon$-stationary Stackelberg equilibrium with convergence guarantees. 
To achieve this, we first reframe the leader–follower interaction as single‑level constrained optimization. 
Second, we define the Lagrangian and show that it can approximate the leader's gradient in response to the equilibrium reached by followers with only first-order gradient evaluations.
These findings suggest a \emph{fully first-order} algorithm that alternates between (i) approximating followers’ best responses through gradient descent and (ii) updating the leader’s strategy via approximating the gradient using Lagrangian. 
Under standard smoothness and strong monotonicity assumptions on the followers’ sub-game, we prove that the algorithm converges to an $\epsilon$‑stationary Stackelberg equilibrium in $O(k^2\epsilon^{-6-\alpha})$ gradient evaluations, where $\alpha>0$ is arbitrarily small. Our method dispenses with any Hessian or matrix‐inverse computations, is scalable to high‑dimensional settings, and provides the first fully first‑order convergence guarantees for multi‑follower Stackelberg games.
\end{abstract}


\section{Introduction}
Stackelberg games model hierarchical interactions where a leader first commits to a decision, anticipating how one or more followers will respond optimally. This structure captures a wide range of real-world systems, including pricing in markets, security resource allocation, and multi-agent learning in strategic settings \citep{
an2017stackelberg,gerstgrasser2023oracles,li2017review}. In Stackelberg games, one agent influences the outcome by acting first, in contrast to simultaneous games where every player chooses their strategy simultaneously. The game’s equilibrium, known as a Stackelberg equilibrium \cite{Bazin2020}, formalizes this anticipatory behavior and has become a fundamental concept in game theory and optimization. In recent years, there has been a surge of interest in computational methods for Stackelberg games, especially in continuous action spaces \citep{fiez2019convergence,mertikopoulos2019learning}, due to their relevance in learning, control, and multi-agent reinforcement learning \citep{groot2017hierarchical,zhang2021multi}.

A common formulation of the Stackelberg equilibrium problem is through bilevel optimization \cite{Bazin2020}: the leader’s objective depends on the followers' equilibrium outcome of a lower-level game played among the followers. This nested structure poses significant computational challenges: the leader must optimize over a followers' equilibrium set that is often not available in closed form and may be sensitive to perturbations. To address this, recent works have proposed alternative formulations such as penalized or regularized versions (e.g., \citep{ji2021bilevel, KKWN23}) that render the problem more tractable by transforming it into a single-level problem. Our work builds on this line of research by introducing a novel Lagrangian-based reformulation for multiple followers, enabling \emph{fully first-order optimization} without requiring exact best-response computations from the followers.

Despite recent advances, existing approaches for computing Stackelberg equilibria often suffer from key limitations that hinder scalability. Many \citep{amos2017optnet,agrawal2019differentiable,WXPRT22} rely on implicit differentiation through the followers’ equilibrium map, which requires computing or inverting large Hessian matrices---a process that is computationally expensive and memory-intensive, especially in high dimensions. Other \citep{yu21sample} assume access to the exact best responses from the followers, which may not be realistic in practice when the lower-level game is non-trivial to solve. These limitations have made prior methods either infeasible for large-scale problems in the multi-follower setting. Our approach circumvents both issues by leveraging a Lagrangian-based reformulation and using only first-order gradient methods throughout, enabling efficient and scalable computation even in complex multi-follower settings.

% \anote{I honestly don't know what else to put in the intro. If I add another summary of the paper it would be too repetitive with the abstract and our contritbution.}
% \kai{This is fine actually. You don't need too much space on other people's work. You can shorten the first three paragraphs.}

\subsection{Our Contributions}
This work addresses the problem of computing Stackelberg equilibria in multi-follower games by leveraging a bilevel optimization framework. Our contributions are as follows:

\begin{itemize}[leftmargin=*]
    \item First, we reformulate the Stackelberg game with multiple followers as a bilevel program and then reduce it to a single-level constrained optimization problem via a Lagrangian penalty formulation. This avoids differentiating through the equilibrium map while retaining the hierarchical structure.
    \item Second, we propose a fully first-order algorithm that alternates between approximating the followers’ equilibrium and updating the leader’s strategy. Importantly, the method relies only on first-order gradients and avoids second-order computations, making it scalable. To the best of our knowledge, we are the first to apply this framework to the multi-follower case. 
    \item Third, we establish convergence guarantees under smoothness and monotonicity assumptions. Our algorithm converges to an $\epsilon$-stationary point of the Stackelberg equilibrium at a rate of $O(k^2\epsilon^{-6-\alpha})$ with precise bounds on both the outer and inner iterations.
\end{itemize}

%Our framework shows that Lagrangian reformulations can unify bilevel game models with practical first-order methods. This perspective not only advances the state of the art for multi-follower Stackelberg games, but also opens new venues for applying fully first-order techniques to broader classes of hierarchical optimization problems.

\subsection{Related Work}
We summarize the related work in: (i) Stackelberg games with single and multiple followers, (ii) bilevel optimization methods, and (iii) first-order algorithms for smooth and monotone games.

\paragraph{Smooth Monotone Games} 
The smoothness and monotonicity of followers’ games ensure existence, uniqueness, and stability of Nash equilibria, enabling analysis via variational inequalities. 
\cite{lin2020finite} establish the first finite-time guarantees for last-iterate convergence, showing that no-regret dynamics converge (rather than just average) to Nash equilibria in monotone games; \cite{GPD} sharpen the result and provide a tight $O(1/\sqrt{T})$ last-iterate rates. 
Recent work has advanced learning dynamics in this setting: \cite{Gao22} establish exponential convergence for continuous-time dynamics; and \cite{tatarenko2019learning} design distributed algorithms that converge under general monotonicity without cost-function knowledge. The most recent result that we are aware is from \cite{Cai23}, where they propose an optimistic accelerated gradient method that achieves $O(1/T)$ last-iterate convergence. We use the result from \cite{Cai23} as a black-box for smooth monotone games.

\paragraph{Stackelberg Games} 
Stackelberg games capture hierarchical leader–follower interactions, with single-follower cases well studied in zero-sum settings \cite{NEURIPS2022_Denizalp,DenizalpNEURIPS2023}. Gradient-based methods (\cite{fiez2020implicit,Jain11}) use implicit optimality conditions, but complexity grows in multi-follower games requiring Nash equilibria at the lower level. \cite{li2022solvingstructuredhierarchicalgames} extend to structured hierarchical games with multiple followers, solving approximate equilibria via back-propagation, while \cite{WXPRT22} similarly apply gradient descent with KKT-based differentiation. Both approaches rely on second-order information, whereas other works (\cite{bacsar,xu2018}) study discrete strategy spaces.

\paragraph{Bilevel Optimization} Bilevel optimization refers to problems where one optimization task (the upper level) is constrained by the solution set of another optimization task (the lower level), making it a natural lens for modeling Stackelberg-type interactions \cite{}. %The first academic formulation of bilevel programming was provided by J. Bracken and J. McGill in 1973 (\cite{Vicente2009}). 

The bilevel perspective has become central in understanding and solving Stackelberg-type problems. \citet{zhang2023introductionbileveloptimizationfoundations} provides a comprehensive overview of this technique from both theoretical and practical perspective. Non-first-order bilevel optimization methods often require implicit differentiation and Hessian-based techniques. For example, \cite{ghadimi2018approximationmethodsbilevelprogramming} design bilevel algorithms that exploit Hessian information of the lower-level to give the first finite-sample complexity guarantees; \cite{ji2021bilevel} efine analysis of implicit/iterative differentiation and propose stocBiO, a stochastic method with efficient Jacobian/Hessian–vector products; \cite{xiao23} extend these ideas to equality-constrained settings with projection-efficient implicit SGD variants achieving near-optimal complexity.

Recent works develop fully first-order bilevel methods that avoid costly second-order information but focus only on single lower-level problems. Here we briefly summarize the convergence rates of first-order bilevel algorithms with \textit{single} lower-level optimization in prior work. 
Under the assumption of smooth bounded gradients/value of $f$ and $g$, \cite{Liu2022} consider the case where the lower-level has unique minimizer, and their algorithm converges to an \textit{$\epsilon$-stationary} point in $\Tilde{O}(\epsilon^{-4})$ iterations. 
\citep{KKWN23,Chen25,yang2023acceleratinginexacthypergradientdescent} further improve the convergence rate to $\Tilde{O}(\epsilon^{-3})$, $\Tilde{O}(\epsilon^{-2})$, and  $\Tilde{O}(\epsilon^{-1.75})$, respectively. Building on \cite{KKWN23}, \cite{maheshwari2024followeragnosticmethodsstackelberg} $O(\epsilon^-2)$ for a single-follower as a direct extension. In this work, we consider $k > 1$ players in the lower while simultaneously optimizing their objectives. 
% \cite{yang2023acceleratinginexacthypergradientdescent} introduces Restarted Accelerated HyperGradient Descent, achieving state-of-the-art guarantees for finding first- and second-order stationary points in nonconvex–strongly-convex bilevel problems.
\cite{LU2024} proposes a first-order quadratic penalty method for bilevel programs, proving convergence to \textit{KKT statioary points} with complexity $\Tilde{O}(\epsilon^{-4})$ without requiring exact Hessians. % \cite{Chen25} develops fully first-order bilevel algorithms (F2BA, AccF2BA) that provably reach near-optimal stationary solutions and outperform existing baselines in practice .

A crucial modeling ingredient in our setting is the strong monotonicity of the followers’ game, which ensures the uniqueness and stability of the lower-level equilibrium. The smoothness and (strongly) convex assumptions are all presented in the aforementioned work.

\section{Preliminaries}
In multi-agent optimization and game theory, Stackelberg games model hierarchical interactions where a \emph{leader} commits to a strategy first, then is followed by \emph{players} or \emph{followers} who respond optimally. These games have received increasing attention for their ability to capture real-world leader–follower dynamics in markets, learning systems, and robust control. A standard and powerful structural assumption is that the followers’ game is smooth and strongly monotone, which admits the uniqueness and stability of the equilibrium response. In particular, monotonicity of the gradient operator allows the use of variational inequality techniques, while smoothness enables efficient algorithmic approximation and stability under perturbations. These assumptions underpin many recent works on equilibria learning in multi-agent games, see \cite{Cai23,GPD,li2020end,tatarenko2019learning}. We build on these developments to propose a fully first-order algorithm for multi-follower Stackelberg games, using the smooth-monotone structure to ensure tractable analysis and convergence guarantees.

\subsection{Smooth Monotone Game}

\begin{definition}
    A multiplayer game is denoted by the tuple $\mathcal{G} = ([k],({Y}_i)_{i\in[k]}, (g_i)_{i\in[k]})$ where: $[k]$ is the set of players, $X_i \in \mathbb R^{n_i}$ is a convex and compact set from which player $i$ chooses their strategy, and $g_i: \mathcal{Y}\to \mathbb{R}$ is the cost function associated with each player such that it takes the input from the set $\mathcal{Y} = \prod_{i = 1}^k Y_i \in \mathbb R^N$ where $N= \sum_{i = 1}^k n_i$.
\end{definition}
Define the gradient operator $V: \mathcal{Y} \to \mathbb R^n$ as $V(\mathbf{y}):= (\nabla_{y_1}g_1(\mathbf y), \nabla_{y_2}g_2(\mathbf y),...,\nabla_{y_k}g_k(\mathbf y))$ where $\mathbf{y} = (y_{i}, y_{-i})$ such that $y_i$ is the strategy chosen by the player $i$ and $y_{-i}$ is the strategy of everyone else.

\begin{definition}[Smooth Monotone Game]
 We say a game is strongly monotone if the gradient operator is strongly monotone, i.e., there exists some $\mu> 0$ such that $\langle V(\mathbf y')- V(\mathbf y), \mathbf y' -\mathbf y\rangle\ge \mu \lVert \mathbf y' -\mathbf y\rVert^2$
    for all $\mathbf y,\mathbf y'\in \mathcal{X}$.
A game is smooth if the gradient operator is smooth with parameter $\mu$, i.e., there exists some $\ell>0$ such that $\langle V(\mathbf y')- V(\mathbf y), \mathbf y' -\mathbf y\rangle\le \ell \lVert \mathbf y' -\mathbf y\rVert^2$
    for all $\mathbf y',\mathbf y\in \mathcal{Y}$.   
\end{definition}

The Jacobian matrix ${DV}(\mathbf y)\in \mathbb{R}^{k\times k}$ of $V$ is defined to be the gradient of $V$:
\begin{equation}
\begin{pmatrix}
\nabla^2_{y_1y_1} g_1(\mathbf y) & \nabla^2_{y_2 y_1} g_1(\mathbf y) & \cdots & \nabla^2_{y_k y_1} g_1(\mathbf y) \\
\nabla^2_{y_1 y_2} g_2(\mathbf y) & \nabla^2_{y_2 y_2} g_2(\mathbf y) & \cdots & \nabla^2_{y_k y_2} g_2(\mathbf y) \\
\vdots & \vdots & \ddots & \vdots \\
\nabla^2_{y_1 y_k} g_k(\mathbf y) & \nabla^2_{y_2 y_k} g_k(\mathbf y) & \cdots & \nabla^2_{y_k y_k} g_k(\mathbf y)
\end{pmatrix}
\end{equation}
Standard variational analysis shows that if $V$ is $\mu$-strongly monotone and $\ell$-smooth, then $\lVert{DV(\mathbf y)}\rVert \le \ell$ and $\lVert DV(\mathbf y)^{-1}\rVert \le 1/\mu$ \citep{facchinei2003finite}, where $\lVert{\cdot}\rVert $ is the spectral norm. Combining with the result of \cite{Cai23} with the strong-monotonicity assumption (\ref{ass:mug})of $\mathcal G$, we obtain an $\epsilon$-gradient guarantee after $O(\mu_g^{-1} \epsilon^{-1})$ of the implicit iterations.

\subsection{Stackelberg Game}
    A Stackelberg game with one leader and $m$ followers can be seen as a two‐stage game where the leader has cost function $f: X\times \mathcal{Y} \to \mathbb R$ and the follower each has cost function $g_i: X \times \mathcal{Y} \to \mathbb R$. The Leader first publicly commits to a strategy $x\in X\subseteq \mathbb R^{n_0},$ then each follower $i= 1,..., k$ simultaneously chooses a strategy $y_i\in Y_i \subseteq \mathbb R^{n_i}, \quad \mathbf{y} = (y_1,..., y_k)\in \mathcal{Y}:= \prod_{i = 1}^k Y_i,$ so as to minimize their own cost function, yielding a simultaneous-move subgame among the followers. 

    Fixing leader's strategy $x$, we denote the followers' subgame by $\mathcal G (x)$. The followers' action $\mathbf y$ is a Nash equilibrium of the subgame $\mathcal G (x)$ if no follower has an incentive to deviate. That is, if we use $\text{NE}(\mathcal G(x))$ to denote the set of Nash equilibria, then $\mathbf y\in \mathcal Y$ satisfies $\forall i,\  g_i(x,y_{i},y_{-i} ) \le g_i((x,y'_{i},y_{-i})\  \forall y'_i\in Y_i.$
    
\begin{definition}[Stackelberg equilibrium]\label{def:stackelberg}
    A Stackelberg equilibrium is a pair $(x^*,\mathbf y^*)$ such that 
    $x^*\in \arg\min_{x\in \mathcal X} \{f(x,\mathbf y): y\in \text{NE}(\mathcal G(x))\}, \text{ where }\mathbf y^* \in \text{NE}(\mathcal G(x^*))$.
\end{definition}

\begin{definition}[$\epsilon$-stationary Stackelberg equilibrium]\label{def:eps}
An $\epsilon$-stationary Stackelberg equilibrium is a pair $(x,\mathbf y)$ such that: (1) fixing leader's current decision $x$, the followers' response is only off by at most $\epsilon$, i.e. $g_i(x,\mathbf y)\le \min_{y'\in Y_i}g(x, y_{i}', y_{-i}) + \epsilon\text{ for all }i$; (2) fixing the followers, response $\mathbf y$, the leader's objective satisfies $\norm{\nabla F(x,\mathbf y)}\le \epsilon$.
    
\end{definition}
In this work, we make the following assumptions, that are standard in the literature:
\begin{assumption}[Followers' subgame strong monotonicity]
    \label{ass:mug} For all leader's strategy $x$, the followers' subgame $\mathcal G$ is strongly monotone with parameter $\mu_g$ and each player's cost function $g_i$ is $\mu_g$-strongly convex in $(x,y_i)$.
\end{assumption}

\begin{assumption}[Smoothness]
     \label{ass:lg} Each follower's cost function $g_i$ and Leader's cost function $f$ are jointly smooth in $(x,\mathbf y)$ with constant $\ell_{g,1}$  and $\ell_{f,1}$, respectively. Furthermore, $g$ is two-times continuously diﬀerentiable, and $\nabla^2g$ is $\ell_{g,2}$-Lipschitz jointly in $(x,\mathbf y)$.
\end{assumption}
\begin{assumption}[Lipschitzness]\label{ass:flipx} 
    $\lVert\nabla_{x} f(x,\overline{\mathbf{y}})\rVert \le \ell_{f,0}$ for all $x$, fixing $\overline{\mathbf{y}}$. $\lVert\nabla_{x} g(x,\overline{\mathbf{y}})\rVert \le \ell_{g,0}$ for all $x$, fixing $\overline{\mathbf{y}}$.
\end{assumption}

Note that Assumption \ref{ass:lg} is equivalent to saying the game $\mathcal G$ is smooth.

\section{Stackelberg Games with Multiple Followers} \label{sc:problem formulation}
The goal is to compute a Stackelberg equilibrium (see Definition \ref{def:stackelberg}) with multiple followers using only first-order information. In a Stackelberg game with one leader and $k$ followers, the leader first commits to a strategy $x \in X$ in the first round, then the followers simultaneously respond with strategy $\mathbf{y}^*(x)$ such that their cost function $g_i(x,y_{i}, y^*_{-i})$ is minimized assuming that everyone else also plays this equilibrium strategy. To simplify the notation, let us introduce an intermediate function $h_i(x, y_i):= g_i(x, y_i, y_{-i}^*(x))$ for each follower $i$. $h_i$ is a function of $x$ and $y_i$ only. It captures the behavior of each follower at equilibrium.

Let $f(x,\mathbf y)$ be the leader's cost function. Define $F(x) = f(x,y_1^*(x), y_2^*(x)..., y_k^*(x))$. We formulate the Stackelberg equilibrium as a bilevel optimization problem:
\begin{gather}\label{problem}
\begin{split}
    \min\nolimits_{x\in X} \quad  F(x) \quad  \text{s.t.}\quad  y_i^*(x) \in \arg\min\nolimits_{y_i\in Y } h_i(x,y_i) \quad \forall i\in [k]   
\end{split}
\end{gather}
Note that (\ref{problem}) is a generalization of the bilevel optimization model of \cite{KKWN23} to $k$ followers. The upper-level problem is the leader's minimization problem, whereas the lower-level problem is to find the follwers' equilibrium for the game $\mathcal{G}(x)$. The upper-level objective is both explicit and implicit in $x$, because $\mathbf{y}^*(x)$ is a solution to the lower-level problem with input $x$. 

If one were to solve (\ref{problem}) via gradient descent, then one necessarily needs to compute the gradient:
\begin{equation}\label{eq:2nd}
    \nabla F(x) =  \nabla_x  f(x,\mathbf{y}^*(x)) +  \sum_{i = 1}^k \nabla_x y_i^*(x)^\top \nabla_{y_i} f(x,\mathbf{y}^*(x)) .
\end{equation}
To obtain $\nabla_x \mathbf y^*(x)$, we first differentiate $\nabla_{y_i} {g_i(x, y_i, y_{-i})}$ with respect to $x$ for all $i$. When evaluating at $\mathbf{y} = \mathbf{y}^*(x)$,  we obtain:
$ \nabla^2_{xy_i}{g_i(x, \mathbf{y}^*(x))} =\nabla^2_{xy_i} g_i(x, \mathbf{y}^*(x)) + \sum_{j =1}^k\nabla^2_{y_jy_i} g_i(x, \mathbf{y}^*(x))\cdot\nabla_{x}y_j^*(x) = 0.$
Writing $H_y =\bigg[\nabla^2_{y_jy_i} g_i \bigg]_{i,j = 1}^k \quad \text{ and } \quad H_x = \bigg[\nabla_{x y_i}^2 g_i\bigg]_{i = 1}^k$ gives us $H_x + H_y \nabla_x \mathbf y^*(x)=0$.
Under strongly-monotone assumption of $\mathcal G$, one shows that $H_y$ is invertible and thus $\nabla_x \mathbf y^*(x) = -H_y^{-1} H_x$ is uniquely determined.

Computing the gradient in (\ref{eq:2nd}) is challenging for two intertwined reasons. First, evaluating the term $\nabla_x  f(x,\mathbf{y}^*(x))$ requires solving the entire followers' subgame to obtain $\mathbf{y}^*(x)$. Second, obtaining the sensitivity $\nabla_x\mathbf{y}^*(x)$ requires differentiating through the equilibrium conditions, which amounts to inverting the Hessian $H_y$. In practice, this “implicit‐function” step requires second‑order information (Hessians and cross‐derivatives) and matrix inversions, making a naïve implementation both computationally and memory prohibitive when the variables are high‑dimensional.

\section{The Fully First-order Method}
This section aims to tackle the two core challenges identified in Section \ref{sc:problem formulation}: (i) the need to solve the followers’ subgame exactly to obtain $\mathbf{y}^*(x)$, and (ii) the reliance on costly second‑order information to differentiate through the equilibrium mapping. To overcome these obstacles, we introduce a Lagrangian reformulation of the bilevel problem, replacing the implicit dependence of the followers' strategies on the leader's decision with a penalized term. Solving the alternative Lagrangian problem corresponds to approximate Stackelberg equilibria without requiring implicit differentiation. Building on this reformulation, we propose a fully first‑order algorithm that alternates between subgame Nash equilibrium update and leader update.

\subsection{Reformulation}
We reformulate (\ref{problem}) so that it becomes a single-level problem with constraints:
\begin{gather}\label{problem2}
\begin{split}
    \min\nolimits_{x\in X, \mathbf{y}\in \mathcal Y} f(x,\mathbf{y}) \quad
    \text{s.t.} \quad h_i(x,y_i)-h_i^*(x)\le 0 \ \forall i\{1,...,k\}
\end{split}
\end{gather}
where $h_i^*(x)= h_i(x,y^*_i(x))= g_i(x,\mathbf{y}^*(x))$. Since (\ref{problem2}) is a constrained optimization program, we can write it as Lagrangian with multiplier $\lambda_1,...,\lambda_k$:
\begin{align*}
\mathcal{L}_{\mathbf{\lambda}} (x,\mathbf{y}) \! := \! f(x,\mathbf{y}) \! + \! \sum\limits_{i=1}^{k}\lambda_i(h_i(x,y_i)-h_i^*(x)) \! = \! f(x,\mathbf{y}) \! + \! \sum\limits_{i=1}^{k}\lambda_i(g_i(x,y_i, {y}_{-i}^*(x))-g_i(x, \mathbf{y}^*(x)))
\end{align*}
In this alternative formulation, since there are $k$ constraints, it involves more Lagrangian terms as opposed to \cite{KKWN23}. Hence, we expect the complexity to be dependent on $k$.

\subsection{Algorithm}

We now provide our full algorithm: see Algorithm~\ref{algo} for a formal description. Before going through the full algorithm, we start by providing the basic intuition that enables our main result.

\paragraph{Intuition} To motivate our algorithm design principle, we begin by showing that the gradient of the true upper-level objective $\nabla F(x)$ can be well approximated using $\nabla_x \mathcal{L}_\lambda(x,\mathbf{y}^*_{\lambda}(x))$, where we define:
\begin{align*}
    & \mathbf{y}_\lambda^*(x) := \arg\min\nolimits_{\mathbf{y}} \mathcal{L}_\lambda(x, \mathbf{y}), 
    & \mathcal{L}_\lambda^*(x):= \min\nolimits_{\mathbf{y}}\mathcal{L}_\lambda(x, \mathbf{y}) = \mathcal{L}_\lambda(x, \mathbf{y}_\lambda^*(x)).
\end{align*}
The minimizer is uniquely defined because, as shown in Lemma~\ref{lagrangian_convex}, the Lagrangian $\mathcal{L}_\lambda$ is strongly convex in $\mathbf{y}$. 
 It follows that $\nabla_{y_i} \mathcal{L}_{\mathbf{\lambda}} (x,\mathbf{y}_\lambda^*(x))=0$ for all $i$. For the Lagrangian to be a good proxy to the true objective, one necessary condition is for the Lagrangian minimizer $y_{i,\lambda_i}^*(x)$ to be close to the true minimizer $y_i^*(x)$ for the lower-level game $\mathcal{G}$. In particular, Lemma \ref{yerr} shows that when $\lambda_i \to \infty$, these two quantities coincide, and $\lVert\nabla F(x) - \nabla \mathcal{L}^*_\lambda(x)\rVert$ in fact goes to 0. %. This in turn implies that $\lVert\nabla F(x) - \nabla \mathcal{L}^*_\lambda(x)\rVert$ is bounded by some constant (see Lemma \ref{bound2}).
\begin{restatable}{lemma}{yerr}\label{yerr}
    For all $i\in [k]$,we have $\lVert y^*_{i,\lambda_i}(x) - y_i^*(x)\rVert\le \frac{2\ell_{f,0}}{\lambda_i\mu_g}$.
\end{restatable}

\begin{lemma}\label{bound2}
    Choosing $\lambda_i = \lambda$ for all $i\in [k]$, we have
    \begin{align*}
        \lVert \nabla F(x)-\nabla \mathcal{L}_\lambda^*(x)\rVert
        &\le k\bigg(\ell_{f,1}+\frac{\ell_{g,1}\ell_{f,1}k}{\mu_g} \bigg)\bigg(\frac{2\ell_{f,0}}{\lambda\mu_g}\bigg) + k\bigg( \lambda\ell_{g,1} + \frac{2\lambda \ell^2_{g,1}}{\mu_g}\bigg) \bigg(\frac{2\ell_{f,0}}{\lambda\mu_g}\bigg)^2.
    \end{align*}
\end{lemma}

\begin{proof} We have $\nabla \mathcal{L}_\lambda^*(x) 
= \nabla_x \mathcal{L}_\lambda(x, \mathbf{y}_\lambda^*(x)) 
+  \sum_{i = 1}^k \nabla_xy_i^*(x)^\top \nabla_{y_i} \mathcal{L}_\lambda(x, \mathbf{y}_\lambda^*(x)) = \nabla_x \mathcal{L}_\lambda(x, \mathbf{y}_\lambda^*(x))$.
The last equality is because $ \nabla_{y_i} \mathcal{L}_\lambda(x, \mathbf{y}_\lambda^*(x)) = 0$ by the optimality of $ \mathbf{y}_\lambda^*(x)$.
Applying Lemma \ref{bound1} with $x$ and $\mathbf y = \mathbf{y}_\lambda^*(x)$, we get:
    \begin{align*}
       & \bigg \Vert \nabla F(x) - \nabla_x \mathcal{L}^*_\lambda(x)\bigg\Vert
        = \lVert \nabla F(x) - \nabla_x \mathcal{L}_\lambda(x, \mathbf{y}_\lambda^*(x))\rVert\\
        %& =\lVert \nabla F(x)-\nabla \mathcal{L}_\lambda^*(x)\rVert\\
        &\le \bigg(\ell_{f,1}+\frac{\ell_{g,1}\ell_{f,1}k}{\mu_g} \bigg)\bigg(\sum_{i = 1}^k\lVert y^*_{i,\lambda_i}(x) - y_i^*(x)\rVert\bigg)   + \bigg( \lambda\ell_{g,1}+ \frac{2\lambda \ell^2_{g,1}}{\mu_g}\bigg) \bigg(\sum_{i = 1}^k\lVert y^*_{i,\lambda_i}(x) - y_i^*(x) \rVert^2\bigg)\\
        &\le k\bigg(\ell_{f,1}+\frac{\ell_{g,1}\ell_{f,1}k}{\mu_g} \bigg)\bigg(\frac{2\ell_{f,0}}{\lambda\mu_g}\bigg) + k\bigg( \lambda\ell_{g,1} + \frac{2\lambda \ell^2_{g,1}}{\mu_g}\bigg) \bigg(\frac{2\ell_{f,0}}{\lambda\mu_g}\bigg)^2.
    \end{align*}
    The last inequality is obtained by applying Lemma \ref{yerr}.
\end{proof}

Lemma \ref{bound2} implies that $\lVert \nabla F(x)-\nabla \mathcal{L}_\lambda^*(x)\rVert \le {k^2C_\lambda}/{\lambda}$ for some constant $C_\lambda$. Thus, if $\lambda\to\infty$, then the difference $\lVert \nabla F(x)-\nabla \mathcal{L}_\lambda^*(x)\rVert$ becomes 0. However, this theorem does not tell us how to obtain $\mathbf{y}_\lambda^*(x)$. This comes from the strong convexity of $ \mathcal{L}_\lambda$ in $\mathbf{y}$, allowing us to use gradient decent to approximate $\argmin_{y_i} \mathcal L(x,  \mathbf{y} )$.
This motivates us to solve the bilevel problem (\ref{problem}) by iteratively solving the alternative formulation (\ref{problem2}).

\paragraph{Algorithm}\footnote{This algorithm can be extended to stochastic setting where noises are presented in upper and lower problems. We expect some blow-up in computation complexity in this case. } We now highlight our full algorithm below:

\begin{algorithm}
\caption{Fully First-order Method for Finding an $\epsilon$-Stackelberg equilibrium with $k > 1$ followers}\label{algo}
\textbf{Input:} $\lambda_0, x_0, [y_{1,0},\dots, y_{k,0}], [z_{1,0},\dots, z_{k,0}]$ \\
\textbf{Output:} 
\begin{algorithmic}[1]
\For{$t = 0, \dots, T-1$}
\State\label{algo:monotonegame} $\mathbf{z}_{t+1} \leftarrow $ solve a $k$-player \textit{strongly monotone game} with player $i$'s objective function $g_i(x_t, \mathbf{z}^t)$. \Statex \Comment{We assume this step takes $M_{z,t}$ gradient steps to solve the strongly monotone game.} % \Comment{Use the current leader strategy $x_t$ to compute the best response $\mathbf z_{t+1}$ for the followers.} 
\State\label{algo:lagragianapp} $\mathbf y_{t+1} \leftarrow \arg\min_{\mathbf{y}_t} \widetilde{\mathcal L}_{\lambda_t} (x,\mathbf y_t, \mathbf z_t) $ within $\epsilon_{y,t}$ accuracy by GD for all $i\in [k]$
\Statex \Comment{We assume this step takes $M_{y,t}$ gradient steps to minimize the Lagrangian.}
\State\label{algo:x} $x_{t+1} \leftarrow x_t- \eta_t \nabla_x \widetilde{\mathcal L}_{\lambda_t} (x,\mathbf y_t, \mathbf z_t)$ \Comment{Update $x_t$ by approx. gradient $\nabla_x \widetilde{\mathcal L}_{\lambda_t}$.}
\State $\lambda_{t+1}\leftarrow \lambda_t + \delta_t$ \Comment{Increase $\lambda_t$ to get more accurate gradient }
\EndFor
\end{algorithmic}
\end{algorithm}

At iteration $t$, the algorithm first takes the leader strategy $x_t$ and the followers, represented by the cost functions $g_i(x_t, \mathbf{y}_{t})$ for all $i$, then find an equilibrium of the strongly-monotone game (approximately) to obtain the follower strategy $\mathbf{z}_{t+1}$ (\cite{GPD}). Then, the algorithm approximates $\mathbf{y}_{\lambda_t}^*(x_t) = \arg\min_{\mathbf{y}} \mathcal{L}_{\lambda_t}(x_t, \mathbf{y}_{t})$ via Gradient Descent. It is important to keep in mind that the algorithm can only obtain the approximated subgame equilibrium $\mathbf{z}$, instead of $\mathbf{y}^*(x)$. Thus, we define $\widetilde{\mathcal L}_{\lambda_t} (x,\mathbf y_t, \mathbf z_t) =f(x,\mathbf{y}_{t})+ \lambda_t\sum_{i=1}^{k}(g_i(x,y_{i,t}, z_{-i,t+1})-g(x, \mathbf{z}_{t+1}))$ to emphasize the fact that $\widetilde{\mathcal L}_{\lambda_t}$ also takes $\mathbf z_t$ at an input. The approximate minimizer gives the next iterate $\mathbf y_{t+1}$. Finally the algorithm updates the leader strategy $x_{t+1}$ with step size $\eta_t$ and the gradient: 
\begin{align*}
    \nabla_x\widetilde{\mathcal L}_{\lambda_t} (x,\mathbf y_t, \mathbf z_t)  
    &= \nabla_xf(x_t,\mathbf{y}_{t}) + \lambda_t\sum_{i=1}^{k} \nabla_x g_i(x_t,y_{i,t+1}, z_{-i,t+1})- \lambda_t\sum_{i=1}^k  \nabla_xg_i(x_t,\mathbf{z}_{t+1}).
\end{align*}
At each time step, the followers best respond to the leader, whose strategy forms a converging sequence to a stationary point.

\section{Convergence Analysis}
We first state the main guarantee of Algorithm \ref{algo}. A proof sketch follows in Subsection \ref{subsc:sketch} and \ref{subsc:err}. 

\begin{restatable}{theorem}{final}\label{final}
    Pick the step size for $\lambda_t$ to be $\delta_t =t^\rho - (t-1)^\rho$ for any $\rho >1$. Then, Algorithm 1 converges to an $\epsilon$-stationary point using at most $O(k^2\epsilon^{-6-\alpha})$ gradient evaluations where $\alpha>0$ is chosen such that $\rho > 1+\alpha/2$.% is the step size parameter for the Lagrangian multiplier $\lambda$.   
\end{restatable}

\subsection{Proof Sketch}\label{subsc:sketch}
In this section, we provide a proof sketch for the main result (Theorem \ref{final}). To start with, we show that using standard gradient descent on the leader's strategy suffices to give small gradient on true objective function $F(x)$. Next, in Subsection \ref{subsc:err}, we show that the surrogate function $\nabla \widetilde{\mathcal L}_{\lambda}$ is a good approximation for $\nabla F(x)$. All the omitted proofs can be found in the Appendix \ref{apdx:5}.
\begin{comment}
\anote{Add more info}

\begin{figure}
    \centering
    \begin{tikzpicture}[node distance=2cm]

\node (start) [process] {Iteration $t$};

\node (step1) [process, below of=start] {Compute the follower best response $\mathbf z_{t+1}$ given $x_t$\\
  \textit{(solve strongly monotone game approximately)}
};

\node (step2) [process, below of=step1] {%
  Approximately minimize the\\ surrogate Lagrangian $\widetilde{\mathcal L}_{\lambda_t}(x_t, \mathbf y_t, \mathbf z_{t+1})$ in $y$\\
  \textit{Do gradient descent until converge}
};

\node (step3) [process, below of=step2] {%
  Update leader strategy:\\
  $x_{t+1} = x_t - \gamma_t \nabla_x \widetilde{\mathcal L}_{\lambda_t}(x_t, \mathbf y_t, \mathbf z_{t+1})$\\
  Increase Lagrangian \\ Multiplier: $\lambda_{t+1} > \lambda_t$
};

% Arrows
\draw [arrow] (start) -> (step1);
\draw [arrow] (step1) -> (step2);
\draw [arrow] (step2) -> (step3);

\end{tikzpicture}
\end{figure}

\end{comment}

\begin{restatable}{theorem}{gradientconverge}\label{gradient_converge}
    Let $\ell_{F,1}$ be the smoothness constant for $F$ 
     (see Lemma \ref{smoothF}). Picking constant step size, $\eta = \frac{1}{\ell_{F,1}}$, we obtain:
    \[
    \sum_{t=0}^T \frac{1}{4\ell_{F,1}}\lVert \nabla F(x_t)\rVert^2 
    \le F(x_0) - F(x^*) +\frac{1}{\ell_{F,1}} \sum_{t=0}^T \lVert err_t \rVert^2,
    \]
    where $\lVert err_t \rVert^2 =\frac{1}{4} \lVert {\nabla_x \widetilde{\mathcal L}_{\lambda_t}(x_t, \mathbf y_t, \mathbf z_{t+1})- \nabla F(x_t)} \rVert^2.$
\end{restatable}

Thus, it suffices to show the cumulative gradient error $\sum_{t=0}^\infty\lVert err_t \rVert^2$  is bounded. Given this, $\lVert \nabla F(x_t)\rVert$ vanishes as $t\to \infty$. It remains to show that we can control $\sum_{t=0}^\infty\lVert err_t \rVert^2 < \infty$ by controlling the implicit inner loops. If this is holds, then Corollary \ref{eps-stationary} shows that we need $T = O(\epsilon^{-2})$ iterations for the outer loop.

\begin{corollary}\label{eps-stationary}
    Suppose $\sum_{t=0}^\infty\lVert err_t \rVert^2$ is bounded and define the constant $C_F = 4\ell_{F,1}(F(x_0)- F(x^*))+ 4\sum_{t=0}^\infty\lVert err_t \rVert^2$. Then, $\min\limits_{0\le t\le T} \lVert \nabla F(x_t) \rVert \le \epsilon$ for any $T\ge C_F/\epsilon^2$.
\end{corollary}

\begin{proof}
    Take the time average for the result in Theorem \ref{outerd} and solve for $T$ gives the result.
\end{proof}
Theorem \ref{gradient_converge}, together with Corollary \ref{eps-stationary}, shows that if the errors decay fast enough, then we can hope to reach the $\epsilon$-stationary point eventually. The next section shows that the premises of Corollary \ref{eps-stationary} can indeed be satisfied.

\subsection{Error decomposition}\label{subsc:err}
The first step of our analysis aims to bound the error between the true gradient of the bilevel objective $\nabla F(x)$ and the approximate gradient $\nabla_x \widetilde{\mathcal L}_{\lambda_t}$ used in the update step. The error comes from 3 places. First, there is a discrepancy using the Lagrangian as a proxy to $F$ (shown in Lemma \ref{bound2}). The other two errors arise from approximating the solution to the strongly monotone game ($\lVert{\mathbf{z_t - \mathbf y^*}}\rVert$) and the minimizer to the Lagrangian ($\lVert{\mathbf{y}_t - \mathbf{y}^*}\rVert$). It shows that \textit {even if} we only have access to the approximated equilibrium of the game $\mathcal G$ at line \ref{algo:monotonegame} and an approximated Lagrangian minimizer, $\nabla \widetilde{L}_{\lambda}$ is still a good proxy to $\nabla F$.

\begin{restatable}{corollary}{outered}\label{outerd}
The following holds at each iteration $t$:
\begin{align}
\begin{split}
    \lVert err_t \rVert^2 \le \underbrace{(\ell_{f,1}^2 + 5k^2 \lambda_t^2) \lVert{\mathbf y_{t+1} - \mathbf y^*_\lambda(x_t)}\rVert^2}_{E_1}  + \underbrace{2 k^2 \lambda_t^2\lVert{\mathbf z_{t+1} - \mathbf y^*(x_t)}\rVert^2}_{E_2} + \underbrace{\frac{k^2C^2_{\lambda}}{\lambda_t^2}}_{E_3}.
    \label{decompose}
\end{split}
\end{align}
where $E_3$ comes direction from Lemma \ref{bound2}.
\end{restatable}
%Corollary \ref{outerd} uses the fact that $\mathcal G$ is strongly-monotone and smooth\textcolor{red}{need cite the assumption}. The proof of the corollary also uses inequality (\ref{cz}) from \cite{GPD}.\jz{This is where we should have the GPD lemma, now here it would make sense. Maybe worth restating, or maybe we can just have a reference to the lemma tying it back to the right notations here}

%\textcolor{red}{I would put Theorem 2.2. after 5.3 and cor 5.4. Saying something in the lines of "notice the error term in theorem 5.3 and corollary 5.4. Then we can bound it using blah, see theorem 5.2---in fact, I would move it to the next section, and dedicate that section to bounding that error term" Also, theorem 5.2 on its own is insufficient to tell the main story because if it still using an endogenous bound that depends on $y_{t+1}$---this is an intermediary result. The theorem here in the main body should give you the final bound that depends only on te parameter of the problem, and not on any of the iterates you introduce/that are artifacts of the algorithm} 

The proof of Corollary \ref{outerd} is provided in Appendix \ref{apdx:missing}. The goal for the rest of the section is to show that the error term, $\lVert err_t \rVert \le E_1 + E_2 + E_3$,
can be arbitrarily small.  Picking $\lambda_t = t^{\rho}$ for $\rho >1$, we get that $\sum_{t= 1}^\infty E_3$ converges. It remains to show $E_1 \le t^{-1 - \epsilon'}$ and $E_2 \le t^{-1-\epsilon'}$ for some $\epsilon'>0$. Note that with this choice of decaying schedule, we have that $\sum_{t=1}^\infty E_1 + E_2 + E_3 < \infty$ . 

To bound $E_1$, we control $\lVert \mathbf y_{t+1} - \mathbf y^*(x)\rVert^2 $ and the step size of $\lambda_t$ at the same time, so that $\mathbf y_{t+1}$ converges to $\mathbf y^*(x)$ faster than $\lambda_t$ grows. Lemma \ref{dy-y*} establishes a recursive relation on $\lVert \mathbf y_{t+1} -\mathbf y^*_{\lambda_t} \rVert$ and gives an upper bound on $\lVert \mathbf y_{t}- \mathbf y^*_{\lambda_{t-1}}\rVert$ . Our proof relies on Assumption \ref{ass:lg} to get strong-convexity and smoothness of $\mathcal {L}_\lambda (x,\mathbf y)$---the corresponding parameters are denoted by $\mu_l$ and $\ell_l$ respectively (see Lemma \ref{lagrangian_convex} and Lemma \ref{smoothl}). Thus, there exists some positive integer ${M_{y,t}}$, the number of implicit iterations at Line \ref{algo:lagragianapp} in Algorithm \ref{algo}, such that the iterate $\mathbf{y}_{t}$ satisfies $\lVert \mathbf y_{t+1} -\mathbf y^*_{\lambda_t}(x_t) \rVert^2 \le \big(1-\frac{2\mu_l}{\mu_l+\ell_l}\big)^{M_{y,t}} \lVert \mathbf y_{t} - \mathbf y^*_{\lambda_t} \rVert^2$, 
This inequality then allows us to build the desired lemma: (The proof is given in Appendix \ref{apdx:missing}.)

\begin{restatable}{lemma}{dyy}\label{dy-y*}
    $\lVert \mathbf y_{t+1} -\mathbf y^*_{\lambda_t} \rVert$ is upper-bounded by 
    \begin{align}\label{lagrangian_aprx}
    \begin{split}
        \lVert \mathbf y_{t+1} -\mathbf y^*_{\lambda_t} \rVert \le \bigg(1-\frac{2\mu_l}{\mu_l+\ell_l}\bigg)^{M_{y,t}/2} \bigg(\lVert \mathbf y_{t} -\mathbf y^*_{\lambda_{t-1}} \rVert + L_{x,t} \eta_{t-1}(\ell_{f,0} + 2k \ell_{g,0}) + L_{\lambda,t}\bigg)
\end{split}
\end{align}
where $L_{x,t} = k\bigg(\frac{2\ell_{f,1} }{\mu_g\lambda_{t+1}}+ \frac{2\ell_{g,1} }{\mu_g}\bigg)$ and $L_{\lambda,t} =  \frac{2k\ell_{f,0} \delta_t}{\mu_g{\lambda_t\lambda_{t+1}}}$. Furthermore, let $\lambda_t = t^{\rho}$ for $\rho > 1$, then $\lVert \mathbf y_{t} -\mathbf y^*_{\lambda_{t-1}} \rVert + L_{x,t} \eta_{t-1}(\ell_{f,0} + 2k \ell_{g,0}) + L_{\lambda,t} \le C_y = O(t^{1-\rho}).$
\end{restatable}

Given Lemma~\ref{dy-y*}, we are ready to bound the term $E_1$ and $E_2$.
\begin{itemize}[leftmargin=*]
    \item \textbf{Bounding $E_1$}: When $\lambda_t \to \infty$, Lemma \ref{dy-y*} says that the term $L_{x,t} \eta_{t-1}(\ell_{f,0} + 2k \ell_{g,0}) + L_{\lambda,t}$ goes to 0. Combined with Lemma \ref{yerr}, which shows that $\mathbf y^*_{\lambda_t} \to \mathbf y^*$, Algorithm \ref{algo} guarantees that $\mathbf{y}_{t}$ is a good proxy for $\mathbf y^*$. In order to ensure $E_1 = O(\| \mathbf y_{t+1} - \mathbf y^*_{\lambda} \|) \le O(t^{-(1+\epsilon')})$, from Equation~\ref{lagrangian_aprx}, we require $M_{y,t}$ the number of GD iterations in Line~\ref{algo:lagragianapp} satisfying $M_{y,t} \ge (\frac{\ell_l}{\mu_l}+1)(\frac{3+\epsilon'}{2}\log t + \log k) = O(\log t)$ for some $\epsilon'>0$.
    \item \textbf{Bounding $E_2$}: Let $V$ be the gradient operator for the lower-level game $\mathcal{G}$ and $M_{z,t}$ be the number of implicit iterations in Line \ref{algo:monotonegame}. \citeauthor{Cai23} show that $\lVert V_{z} (x,\mathbf z_{M_{z,t}})\rVert \le \frac{C_z}{{M_{z,t}}} $, where $C_z = O (k)$ is linear in $k$ with constants depend the parameter of $\mathcal G$ and the distance between the initial point $z_0$ and $z^*$.  By the strong monotonicity of $\mathcal G$ (Assumption \ref{ass:mug}), we obtain $\lVert \mathbf{z}_{M_{z,t}}- \mathbf{z}^*\rVert = \lVert \mathbf{z}_{t+1} - \mathbf{y}^*(x) \rVert \le \frac{C_z}{\mu_g\sqrt{M_{z,t}}}.$ Thus, $E_2$ can be driven arbitrarily small by setting $M_{z,t}$ large enough. Formally, $E_2\le t^{-(1+\epsilon')}$ for $M_{z,t}\ge \frac{C_z k t^{\rho + \epsilon'+1}}{\mu_g} $.
\end{itemize}

Finally, we are ready to prove Theorem \ref{final}.
\begin{proof}[Proof of Theorem~\ref{final}]
    For any $\rho>1$, pick $\alpha$ and $\epsilon'$ such that $\frac{\alpha}{2} > \epsilon' >0$. We set $\frac{\alpha}{2}- \epsilon' = \rho- 1 $. By the above argument, it suffices to set $M_{y,t}\ge O(\log t)$ and $M_{z,t}\ge O(k^2 t^{2+\frac{\alpha}{2}})$. With these choices of iteration complexity for the inner loop at Line \ref{algo:lagragianapp} and Line \ref{algo:monotonegame} of Algorithm \ref{algo}, respectively, the cumulative gradient error in Theorem~\ref{gradient_converge} is bounded by a constant. 
    
    Lastly, Corollary \ref{eps-stationary} suggests that we need a total of $T = O(\epsilon^{-2})$ iterations to reach an $\epsilon$-stationary point. Thus, summing $T$ iterations of $M_{z,t}$ and $M_{y,t}$ gives a total of $\sum_{t=0}^T M_{z,t} = O(k^2T^{3+\alpha/2}) =  O(k^2\epsilon^{-6-\alpha})$
    iterations to solve the strongly monotone game and a total of 
    $\sum_{t=0}^T M_{y,t} = O(T\log (kT))  = O(\epsilon^{-2}\log (k\epsilon^{-2}) ))$
    iterations to approximate the Lagrangian minimizer. 
\end{proof}

\paragraph{Conclusion} This work distinguishes itself by delivering the first algorithm for multiple-follower Stackelberg problems that requires only first-order oracles, yet still carries a provable $O(k^2\epsilon^{-6-\alpha})$ gradient-evaluation bound for reaching an $\epsilon$-stationary equilibrium where $\alpha>0$. Current works that rely only on fully first-order method do not apply to Stackelberg games with multiple followers, see \cite{Jain11,ji2021bilevel}. Whereas in works that do focus on multiple follower settings (\cite{li2020end,li2022solvingstructuredhierarchicalgames,WXPRT22}), their convergence guarantee crucially requires second-order implicit differentiation.

\paragraph{Future Work}Several directions remain open. First, relaxing structural assumptions like strong monotonicity could broaden applicability. Second, lower bounds on the complexity of computing multi-follower equilibria would clarify whether our current rate is optimal. Third, reducing dependence on the number of followers $k$ could accelerate convergence. Finally, better choices of the penalty parameter 
$\alpha$ and more efficient inner-loop solvers may yield substantial computational gains.

\bibliographystyle{plainnat}
\bibliography{ref}

\newpage
\appendix
\section{Missing Proofs}\label{apdx:missing}

\yerr*

\begin{proof}
    Take Lemma \ref{dy*t} and let $\lambda_{2,i} \to \infty$ and hence $\lambda_2\to \infty$: $\lim_{\lambda_{2,i} \to \infty} y^*_{\lambda_{2,i}}(x) = y^*_i(x)$. This yield the result.
\end{proof}

\gradientconverge*
\begin{proof}
    {The statement follows directly from the proof of Lemma \ref{truegradient}}
\end{proof}

\outered*

\begin{proof}\label{outered}
The statement follows directly from the proof of Lemma \ref{truegradient}.
\end{proof}

\dyy*
\begin{proof}
    By the strong convexity and smoothness of $\mathcal L$, we have $$\lVert \mathbf y_{t+1} -\mathbf y^*_{\lambda_t} \rVert \le \bigg(1-\frac{2m_t}{m_t+L_t}\bigg)^{M_{y,t}/2} \lVert \mathbf y_{t} -\mathbf y^*_{\lambda_t} \rVert. $$ We first bound $\lVert \mathbf y_{t} -\mathbf y^*_{\lambda_t} \rVert$ using triangle inequality and introducing $\mathbf y^*_{\lambda_{t-1}}$:
    \begin{align*}
        \lVert \mathbf y_{t} -\mathbf y^*_{\lambda_t} \rVert  \le  \lVert \mathbf y_{t} -\mathbf y^*_{\lambda_{t-1}} \rVert +  \lVert \mathbf y^*_{\lambda_{t-1}} -\mathbf y^*_{\lambda_t} \rVert.
    \end{align*}
    Hence, we may bound each sum separately. By Lemma \ref{dy*t}, we can bound $\lVert \mathbf y^*_{\lambda_{t-1}} -\mathbf y^*_{\lambda_t} \rVert $ as follows: 
    \[\lVert \mathbf y^*_{\lambda_{t-1}} -\mathbf y^*_{\lambda_t} \rVert \le L_{x,t} \lVert x_t- x_{t-1} \rVert + L_{\lambda,t}.\]
    Replacing $\lVert x_t- x_{t-1} \rVert$ with the upper bound given in Lemma \ref{dx} gives the inequality (\ref{lagrangian_aprx}).

    Now we show the second part of the statement. 
    Let us introduce auxiliary symbols. Let 
    \begin{align*}
        \lVert \mathbf y_{t+1} -\mathbf y^*_{\lambda_t} \rVert &\le \bigg(\underbrace{1-\frac{2m_t}{m_t+L_t}}_{=:q_t}\bigg)^{M_{y,t}/2} \bigg(\underbrace{\lVert \mathbf y_{t} -\mathbf y^*_{\lambda_{t-1}} \rVert}_{=:R_{t-1}} + \underbrace{L_{x,t} \eta_{t-1}(\ell_{f,0} + 2k \ell_{g,0}) + L_{\lambda,t}}_{=:B_t}\bigg).
    \end{align*}
    Then, inequality (\ref{lagrangian_aprx}) can be written as the recursive relation 
    \begin{equation}\label{recursive}
        R_{t} \le q_t^{M_{y,t}} R_{ t-1}+ q_t^{M_{y,t}}B_t.
    \end{equation}
    Define $\theta_{t+1}:= R_{t}+ B_{t+1}$.  We want to show $\theta_{t+1} \le C_y$ for some constant $C_y$. We can expand $\theta_{t+1}$ using (\ref{recursive}) as \[\theta_{t+1} =R_{t}+ B_{t+1} \le q_t^{M_{y,t}} R_{ t-1}+ q_t^{M_{y,t}}B_t + B_{t+1}.\]
    Let $\Theta_{t+1}:= \max_{i} \theta_{t+1} $, it follows that 
    \begin{align*}
       \theta_{t+1} &= q_t^{M_{y,t}}  \theta_t + B_{t+1}\\
       &= \prod_{s=0}^t q_t^{M_{y,s}}\theta_t + \sum_{u=1}^{t+1} \bigg(\prod_{s=u}^t q_{s}^{M_{y,s}}\bigg)B_u\\
       &\le \theta_0 + \sum_{u= 1}^{t+1}B_u= :C_y
    \end{align*}

    Choosing $\delta_t = \delta_t = t^\rho - (t-1)^{\rho}$, we conclude that $L_{x,t} = O(\lambda_{t+1}^{-1}) = O(t^{-\rho})$ and $L_{\lambda,t} = O(\frac{\delta_t}{\lambda_t \lambda_{t+1}}) = O(t^{-\rho})$ are both geometric series.  Thus, \[B_t = \underbrace{ L_{x,t} \eta_{t-1}(\ell_{f,0} + 2k \ell_{g,0})}_{O(t^{-\rho})
    )} + \underbrace{L_{\lambda,t}}_{O(t^{-\rho})}\]
    is a geometric series. Hence the sum $\sum_{u= 1}^{t+1} B_u = O(t^{1-\rho})$.
\end{proof}

%\rateMyt*
%\begin{proof}
    %Using the notation and results from Thereon \ref{outerd} and Lemma \ref{dy-y*}, observe $$A_t  \lVert \mathbf y_{t+1} - \mathbf y^*
   %\rVert \le A_tq_t^{M_{y,t}}C_y.$$ Thus requiring $A_tq_t^{M_{y,t}}C_y\le \epsilon_{y,t}$ gives $M_{y,t}\ge \frac{m_t+L_t}{2m_t} \log\bigg(\frac{A_t C_y}{\epsilon_{y,t}}\bigg)$
%\end{proof}

%\rateMzt*
%\begin{proof}
    %Requiring $\frac{\lambda_t\ell_{g,1}kC_z}{\mu_g \sqrt{M_{z,t}}} \le \epsilon_{z,t}$ is equivalent to $M_{z,t}\ge \bigg(\frac{\lambda_t\ell_{g,1}kC_z}{\mu_t\epsilon_{z,t}}\bigg)^2$.
%\end{proof}

\section{Auxiliary Lemmas}
\subsection{Auxiliary Lemmas for Section 4} \label{apdx:4}
\begin{lemma}\label{bound1} 
    For any $x$, $\mathbf{y}$, $\lambda_i = \lambda$ for all $i$, the following holds:
    \begin{align*}
        &\bigg \Vert \nabla F(x) - \nabla_x \mathcal{L}_\lambda(x,\mathbf{y}) - \sum_{i=1}^k \nabla_x y_i^*(x)^\top \nabla_{y_i} \mathcal{L}_{\mathbf{\lambda}}(x,\mathbf{y})\bigg\Vert
        \\&\le \bigg(\ell_{f,1}+\frac{\ell_{g,1}\ell_{f,1}k}{\mu_g} \bigg)\bigg(\sum_{i = 1}^k\lVert y_i- y_i^*(x) \rVert\bigg) + \bigg( \lambda\ell_{g,1} + \frac{2\lambda \ell^2_{g,1}}{\mu_g}\bigg) \bigg(\sum_{i = 1}^k\lVert y_i- y_i^*(x) \rVert^2\bigg).
    \end{align*}
\end{lemma}
\begin{proof}
    We first write out the (partial) derivatives for $F(x)$: 
    \[\nabla F(x) =\nabla_x F(x) =\nabla_x  f(x,\mathbf{y}^*(x)) +  \sum_{i = 1}^k \nabla_x y_i^*(x)^\top \nabla_{y_i} f(x,\mathbf{y}^*(x)). \]

    We first write out the (partial) derivatives for the Lagrangian:
    \begin{align*}
        \nabla_x  \mathcal{L}_{\mathbf{\lambda}} (x,\mathbf{y}) &=  \nabla_x  f(x,\mathbf{y}) + \sum_{i=1}^{k} \lambda_i\nabla_x  g_i(x,y_i, y_{-i}^*(x)) -\sum_{i=1}^k \lambda_i\nabla_x g_i(x,\mathbf{y}^*(x))\\
        &+\sum_{i = 1}^k \lambda_i \nabla_x y_{-i}^*(x)^\top \nabla_{y_{-i}}g_i(x, y_i, y_{-i}^*(x)) - \sum_{i = 1}^k \lambda_i \nabla_x \mathbf y^*(x)^\top \nabla_{\mathbf y}g_i(x,\mathbf y^*(x))\\
        &= \nabla_x  f(x,\mathbf{y}) + \sum_{i=1}^{k} \lambda_i \bigg(\nabla_x g_i(x,y_i, y_{-i}^*(x))- \nabla_x g_i(x,\mathbf{y}^*(x))\bigg)\\
        &+ \sum_{i =1}^k\lambda_i \sum_{j\neq i}\nabla_x y_{j}^*(x)^\top \bigg(\nabla_{y_{j}}g_i(x, y_i, y_{-i}^*(x)) - \nabla_{y_{j}} g_i(x,\mathbf y^*(x)) \bigg).
    \end{align*}
   The last equality uses the fact that $\nabla_{y_i} g_i(x, y^*_i, y^*_{-i}) = 0$, so only the off-diagonal terms are left. 
   
   By the first-order condition, we have
    
    \begin{align}
        { \nabla_x y^*(x) = \begin{bmatrix} \nabla_x y^*_1(x) \\
        \nabla_x y^*_2(x) \\
        \cdots \\
        \nabla_x y^*_k(x)
        \end{bmatrix} = - [\nabla_{y_j y_i}^2 g_i(x,\mathbf{y}^*)]_{i,j \in [k]}^{-1} [\nabla_{y_i x}^2 g_i(x,\mathbf{y}^*)]_{i\in[k]} = - H_y^{-1} H_x}.
    \end{align}
    Moreover, 
    \[\nabla_{y_i}\mathcal{L}_{\mathbf{\lambda}} (x,\mathbf{y}) = \nabla_{y_i}f(x,\mathbf{y}) + \lambda_i \nabla_{y_i}g_i(x,y_i, y^*_{-i}(x)) \ \forall i,
    \]
    which implies
    Substituting each term into the left-hand side and rearranging, we obtain:
    \begin{align*}
        & \bigg\lVert \nabla F(x) - \nabla_x \mathcal{L}_\lambda(x,\mathbf{y}) - \sum_{i=1}^k \nabla_x y_i^*(x)^\top \nabla_{y_i} \mathcal{L}_{\mathbf{\lambda}}(x,\mathbf{y})\bigg\rVert \\
        =& \bigg \lVert \underbrace{\biggl(\nabla_x  f(x,\mathbf{y}^*(x)) - \nabla_x  f(x,\mathbf{y})\biggr)}_{a} +\biggl(\sum_{i=1}^{k} \lambda_i \underbrace{\biggl( \nabla_x g_i(x,\mathbf{y}^*(x))-\nabla_x  g_i(x,y_i, y_{-i}^*(x)) { + \nabla_{x y_i} g_i(x, y^*_i, y^*_{-i}) (y_i - y^*_i)} \biggr)}_{b}\biggr)
        \\
        & -  \sum_{i = 1}^k { \nabla_{x} y^*_i(x)^\top} \biggl( \lambda_i \underbrace{\left( \nabla_{y_i}g_i(x,y_i, y^*_{-i}(x)) { - \nabla_{y_i} g_i(x, y^*_i, y^*_{-i}) - \nabla_{y_i y_i} g_i(x, y^*_i, y^*_{-i}) (y_i - y^*_i)}\right)}_{c}\\
        & + \underbrace{\nabla_{y_i} f(x,\mathbf{y})-\nabla_{y_i} f(x,\mathbf{y}^*(x))}_{d} \biggr) \\
        &- \underbrace{{\bigg( \sum_{i=1}^k \lambda_i \left( \nabla_{x y_i} g_i(x,y^*_i, y^*_{-i}) (y_i - y^*_i) + {\nabla_x y^*_i(x)}^\top \nabla_{y_i y_i} g_i(x,y^*_i, y^*_{-i}) (y_i - y^*_i) \right)\bigg)}}_{e} \\
        &{- \underbrace{\sum_{i=1}^k \lambda_i \sum_{j\neq i}\nabla_x y_{j}^*(x)^\top \bigg(\nabla_{y_{j}}g_i(x, y_i, y_{-i}^*(x)) - \nabla_{y_{j}} g_i(x,\mathbf y^*(x)) \bigg)}_{f}}\bigg\rVert.
    \end{align*}

Let's decompose the sum and bound each term. Starting with the last two terms:
\begin{align*}
    e &= \lambda\sum_{i = 1}^k\left( \nabla_{x y_i} g_i(x,y^*_i, y^*_{-i}) (y_i - y^*_i) + {\nabla_x y^*_i(x)}^\top \nabla_{y_i y_i} g_i(x,y^*_i, y^*_{-i}) (y_i - y^*_i) \right)\\
    & = \lambda \sum_{i = 1}^k (y_i - y^*_i)\underbrace{\bigg(\nabla^2_{xy_i} g_i(x, y_i^*, y^*_{-i}) + \sum_{j =1}^k\nabla_{x}y_j^*(x)^\top \nabla^2_{y_jy_i} g_i(x, y^*_i, y^*_{-i})\bigg) }_{=0 \text{ by }(\ref{eq:2nd})} \\
    & \quad - \lambda\sum_{j\neq i}\nabla_{x}y_j^*(x)^\top \nabla^2_{y_jy_i} g_i(x, y_i^*, y^*_{-i})(y_i - y^*_i)\\
    &= - \lambda\sum_{j\neq i}\nabla_{x}y_j^*(x)^\top \nabla^2_{y_jy_i} g_i(x, y^*_i, y^*_{-i})(y_i - y^*_i).
\end{align*}

%$$\norm{e} \le \lambda\sum_{j\neq i} \norm{\mathbf y^* (x)} \norm{\nabla^2_{y_jy_i} g_i(x, y^*_i, y^*_{-i})} \norm{y_i - y^*_i} \le \frac{k \lambda\ell_{g,1}^2}{\mu_g}\norm{y_i - y^*_i}$$
Now combining $e$ and $f$ gives us:
\begin{align*}
    e+f = \lambda \sum_{j\neq i}\nabla_{x}y_j^*(x)^\top\bigg(\nabla_{y_{j}}g_i(x, y_i, y_{-i}^*(x)) - \nabla_{y_{j}} g_i(x,\mathbf y^*(x))- \nabla^2_{y_jy_i} g_i(x, y^*_i, y^*_{-i})(y_i - y^*_i)\bigg).
\end{align*}
It follows by Assumption \ref{ass:lg} (smoothness) that:
\[\lVert\nabla_{y_{j}}g_i(x, y_i, y_{-i}^*(x)) - \nabla_{y_{j}} g_i(x,\mathbf y^*(x))- \nabla^2_{y_jy_i} g_i(x, y^*_i, y^*_{-i})(y_i - y^*_i)\rVert \le \ell_{g,1}\lVert y_i - y_i^*(x)\rVert^2.\]
By smoothness and strong monotonicity assumption (\ref{ass:mug}), we also have
\[\bigg\lVert \nabla_x \mathbf y^* (x)\bigg\rVert = \lVert-H_y^{-1} H_x \rVert\le \frac{\ell_{g,1}}{\mu_g} \implies \bigg\lVert
 \nabla_x y^*_i(x)\bigg\rVert \le  \frac{\ell_{g,1}}{\mu_g}.\]
Thus, 
\[\norm{e+f} \le \frac{\lambda \ell^2_{g,1}}{\mu_g} \sum_{i= 1}^k\lVert y_i - y_i^*(x)\rVert^2.\]

By Assumption \ref{ass:lg}---specifically the smoothness of $f$ in $(x,\mathbf{y})$---, we obtain:
\[
\lVert{a}\rVert\le \ell_{f,1} \lVert\mathbf{y} - \mathbf{y}^*(x)\rVert \le \ell_{f,1} \sum_{i = 1}^k\lVert y_i- y_i^*(x) \rVert. 
\]
By Assumption \ref{ass:lg}---specifically the smoothness of $g$ in $(x,y_i)$---, we have:
\[\lVert{b}\rVert \le \ell_{g,1}\lVert y_i - y_i^*(x)\rVert^2.\]
By Assumption \ref{ass:lg}---specifically the smoothness of $g$ in $(x,y_i)$:
\begin{align*}
    \lVert{c}\rVert
    &\le \ell_{g,1}\lVert y_i - y_i^*(x)\rVert^2.
\end{align*}
By Assumption \ref{ass:lg}---specifically the smoothness of $f$ in $(x,\mathbf{y})$:
\[
\lVert{d}\rVert\le \ell_{f,1} \lVert\mathbf{y} - \mathbf{y}^*(x)\rVert \le \ell_{f,1} \sum_{i = 1}^k\lVert y_i- y_i^*(x) \rVert. 
\]

%Here we use the assumption that, for all $i$, $g_i(x,y_i,y_{-i}^*(x))$ is $\ell_{g,1}$ smooth in $(x,y_i)$, and $f(x,\mathbf{y})$ is $\ell_{f,1}$ smooth in $(x,\mathbf{y})$.

Putting everything together, we obtain:
\begin{align*}
    &\bigg \lVert \nabla F(x) - \nabla_x\mathcal{L}_\lambda(x,\mathbf{y}) - \sum_{i=1}^k \nabla_xy_i^*(x)^\top \nabla_{y_i} \mathcal{L}_{\mathbf{\lambda}}(x,\mathbf{y})\bigg\rVert \\
    &\le \lVert{a}\rVert + \sum_{i=1}^k \lambda_i \lVert{b}\rVert + \sum_{i =1}^k\frac{\ell_{g,1}}{\mu_g} \bigg(\lambda_i\lVert{c}\rVert + \lVert{d}\rVert\bigg) + \lVert e+f\rVert\\
    &\le \ell_{f,1} \sum_{i = 1}^k\lVert y_i- y_i^*(x) \rVert 
    +  \sum_{i = 1}^k \lambda_i\ell_{g,1}\lVert y_i - y_i^*(x)\rVert^2  + \sum_{i = 1}^k \frac{\ell_{g,1}}{\mu_g}\bigg(\lambda_i\ell_{g,1}\lVert y_i - y_i^*(x)\rVert^2+\ell_{f,1} \sum_{i = 1}^k\lVert y_i- y_i^*(x) \rVert \bigg) \\
    & \quad \quad +\frac{\lambda \ell^2_{g,1}}{\mu_g} \sum_{i= 1}^k\lVert y_i - y_i^*(x)\rVert^2 \\
    & \le \bigg(\ell_{f,1}+\frac{\ell_{g,1}\ell_{f,1}k}{\mu_g} \bigg)\bigg(\sum_{i = 1}^k\lVert y_i- y_i^*(x) \rVert\bigg) + \bigg( \lambda\ell_{g,1} + \frac{2\lambda \ell^2_{g,1}}{\mu_g}\bigg) \bigg(\sum_{i = 1}^k\lVert y_i- y_i^*(x) \rVert^2\bigg).
\end{align*}
The last inequality is by assuming $\lambda_i = \lambda$ for all $i\in[k]$.
\end{proof}

\begin{comment}
    \begin{lemma}
    $\mathbf{y}^*_{\lambda}(x)$ is $\ell_{\lambda,1}$ smooth in $x$.
\end{lemma}
\end{comment}

\subsection{Auxiliary Lemmas for Section 5} \label{apdx:5}
\begin{lemma}
    $\mathbf y^*(x)$ is $\frac {\ell_{g,1}}{\mu_g}$-smooth in $x$.
\end{lemma}
\begin{proof}
    It follows from the proof of Lemma \ref{bound1}.
\end{proof}

\begin{lemma}\label{smoothF}
    The smoothness constant for $F$ is given by
    $\ell_{F,1} = (\ell_{f,1} +\frac{\ell_{f,0}\ell_{g,2}}{\mu_g} +  \frac{\ell_{g,1} \ell_{f,1}}{\mu_g}) (1+\frac {\ell_{g,1}}{\mu_g})$
\end{lemma}
\begin{proof}
    Recall equation (\ref{eq:2nd}). We begin by writing 
    \[\nabla F(x) = \nabla_x f(x,\mathbf y^*(x)) - \nabla_x \mathbf y^*(x)^\top \nabla_{\mathbf y} f(x,\mathbf y^*(x)).\]
    Define $J(x): =\nabla_x \mathbf y^*(x)^\top =(- H_y^{-1}H_x)^{\top}$, and $\alpha : =\frac {\ell_{g,1}}{\mu_g} $. Note that in the proof of Lemma \ref{bound1} we establish that $\lVert J(x)\rVert \le \alpha$. Consider $\nabla F(x_1) - \nabla F(x_2)$, we have
    \begin{align}\label{eq:Fsmooth}
    \begin{split}
        \lVert \nabla F(x_1) - \nabla F(x_2)\rVert  &\le  \underbrace{\lVert \nabla_x f(x_1,\mathbf y^*(x_1)) - \nabla_x f(x_2,\mathbf y^*(x_2))\rVert}_{(a)} \\
        & \quad + \underbrace{\lVert J(x_1)^\top \nabla_{\mathbf y} f(x_1,\mathbf y^*(x_1)) - J(x_2)^\top\nabla_{\mathbf y} f(x_2,\mathbf y^*(x_2))\rVert}_{(b)}.
    \end{split}
    \end{align}
    Term $(a)$ can be bounded using the smoothness of $f$ (Assumption \ref{ass:lg}) and the smoothness of $\mathbf y^*$ in $x$. Concretely:
    \begin{align*}
        (a)&\le \ell_{f,1} \lVert (x_1, \mathbf y^*(x_1)) - (x_2, \mathbf y^*(x_2))\rVert \le \ell_{f,1} \sqrt{(1+\alpha^2)}\lVert x_1 - x_2\rVert \le  \ell_{f,1}{(1+\alpha)}\lVert x_1 - x_2\rVert.
    \end{align*}
    Now let's consider term $(b)$. We can rewrite it as 
    \begin{align*}
        (b) & = \lVert (J(x_1)^\top  -  J(x_2)^\top )\nabla_{\mathbf y}f(x_1,\mathbf y^*(x_1)) +  J(x_2)^\top \nabla_{\mathbf y} f(x_1,\mathbf y^*(x_1))-J(x_2)^\top\nabla_{\mathbf y} f(x_2,\mathbf y^*(x_2))\rVert\\
        & = \bigg\lVert (J(x_1)^\top  -  J(x_2)^\top )\nabla_{\mathbf y}f(x_1,\mathbf y^*(x_1)) +  J(x_2)^\top \bigg(\nabla_{\mathbf y} f(x_1,\mathbf y^*(x_1))-\nabla_{\mathbf y} f(x_2,\mathbf y^*(x_2))\bigg)\bigg\rVert\\
        & \le \lVert J(x_1)^\top  -  J(x_2)^\top\rVert \lVert \nabla_{\mathbf y}f(x_1,\mathbf y^*(x_1)) \rVert + \lVert J(x_2)^\top \rVert \lVert \nabla_{\mathbf y} f(x_1,\mathbf y^*(x_1))-\nabla_{\mathbf y} f(x_2,\mathbf y^*(x_2))\rVert \\
        & \le \ell_{f,0} \lVert \nabla_x J(z)\rVert \lVert x_1-x_2 \rVert + \ell_{f,1} \alpha (1+\alpha) \lVert x_1 - x_2\rVert,
    \end{align*}
    where in the last inequality we use the Mean Value Theorem with $z\in[x_1, x_2]$ to bound $\lVert J(x_1)^\top  -  J(x_2)^\top\rVert$, the Lipchitzness of $f$ using Assumption \ref{ass:flipx} and the smoothness of $f$ using Assumption \ref{ass:lg}.

    It remains to bound $\lVert \nabla_x J(z)\rVert$. Here using Assumption \ref{ass:lg} again, where we assumed that $\nabla^2g$ is Lipschitz, we have \[\lVert \nabla_x J(z) \rVert \le \frac{\ell_{g,2}}{\mu_g}(1+\alpha)\]

    Now put everything back into (\ref{eq:Fsmooth}), we obtain,
    \begin{align*}
        \lVert \nabla F(x_1) - \nabla F(x_2)\rVert&\le (a) + {(b)}\\
        &\le \ell_{f,1}{(1+\alpha)}\lVert x_1 - x_2\rVert + \frac{\ell_{f,0}\ell_{g,2}}{\mu_g}(1+\alpha) \lVert x_1-x_2 \rVert + \ell_{f,1} \alpha (1+\alpha) \lVert x_1 - x_2\rVert\\
        &= (\ell_{f,1} +\frac{\ell_{f,0}\ell_{g,2}}{\mu_g} +  \ell_{f,1} \alpha) (1+\alpha) \lVert x_1 - x_2\rVert.
    \end{align*}
\end{proof}

\begin{lemma}\label{smoothl}
The Lagrangian $\mathcal{L}_{\lambda}(x,\mathbf y)$ is $\ell_{l} : = \ell_{f,1} + k\lambda \ell_{g,1}$ smooth in $\mathbf y$.
\end{lemma}

\begin{proof}
    Consider the Hessian of $\mathcal L$ with respect to $\mathbf y$. We note that:
    \begin{align*}
        \norm{\nabla_{yy} \mathcal{L}(x,\mathbf y)} &\le \norm{\nabla_{yy} f(x,\mathbf y)} + \lambda\sum_{i = 1}^k  \norm{\nabla_{y_i,y_i} g_i(x,y_{i}, y_{-i}^*(x))}\le \ell_{f,1} + k\lambda \ell_{g,1}.
    \end{align*}
\end{proof}

\begin{lemma}\label{truegradient}
    Let $\ell_{F,1}$ be the smoothness constant for $F$. Then for any two iterates $x_t$ and $x_{t+1}$. we have: 
\begin{align*}
    \frac{\eta_t}{2} \norm{\nabla F(x_t)}^2 \le F(x_t) - F(x_{t+1}) + {2\eta_t} \bigg((\ell_{f,1}^2 + k^2 \lambda_t^2) \norm{\mathbf y_{t+1} - \mathbf y^*_\lambda(x_t)}^2 + 2 k^2 \lambda_t^2\norm{\mathbf z_{t+1} - \mathbf y^*(x_t)}^2 + \frac{k^2C_{\lambda_t}^2}{\lambda_t}\bigg).
\end{align*}
\end{lemma}
\begin{proof}
    By the smoothness of $F$ in $x$, we have:
    \begin{align}\label{F-F}
    \begin{split}
        F(x_{t+1} - F(x_t)) &\le \langle \nabla F(x_t), x_{t+1} - x_t \rangle + \frac{\ell_{F,1}}{2} \norm{x_{t+1}- x_t}^2\\
        &= -\eta_t \langle \nabla F(x_t), \nabla_{x} \widetilde{\mathcal L}_{\lambda_t}(x_t, \mathbf y_t, \mathbf z_{t+1})\rangle + \frac{\ell_{F,1}\eta_t^2}{2} \norm{\nabla_x \widetilde{\mathcal L}_{\lambda_t}(x_t, \mathbf y_t, \mathbf z_{t+1})}^2\\
        &= -\frac{\eta_t}{2}\bigg( \norm{\nabla F(x_t)}^2 + \norm{\nabla_x \widetilde{\mathcal L}_{\lambda_t}(x_t, \mathbf y_t, \mathbf z_{t+1})}^2 - \norm{\nabla F(x_t) - \nabla_x \widetilde{\mathcal L}_{\lambda_t}(x_t, \mathbf y_t, \mathbf z_{t+1})}^2\bigg)\\
        &+ \frac{\ell_{F,1}\eta_t^2}{2} \norm{\nabla_x \widetilde{\mathcal L}_{\lambda_t}(x_t, \mathbf y_t, \mathbf z_{t+1})}^2\\
        &=-\frac{\eta_t}{2} \norm{\nabla F(x_t)}^2  + \frac{\eta_t}{2}\norm{\nabla F(x_t) - \nabla_x \widetilde{\mathcal L}_{\lambda_t}(x_t, \mathbf y_t, \mathbf z_{t+1})}^2 -\frac{\eta_t}{4} \norm{\nabla_x \widetilde{\mathcal L}_{\lambda_t}(x_t, \mathbf y_t, \mathbf z_{t+1})}^2\\
        &\le -\frac{\eta_t}{2} \norm{\nabla F(x_t)}^2  + \frac{\eta_t}{2}\norm{\nabla F(x_t) - \nabla_x \widetilde{\mathcal L}_{\lambda_t}(x_t, \mathbf y_t, \mathbf z_{t+1})}^2.       
    \end{split}
    \end{align}

    We now bound the $\nabla F(x_t) - \nabla_x \widetilde{\mathcal L}_{\lambda_t}(x_t, \mathbf y_t, \mathbf z_{t+1})$ term. To do so, note that
    \begin{equation}\label{eq:q-f}
        \nabla_x \widetilde{\mathcal L}_{\lambda_t}(x_t, \mathbf y_t, \mathbf z_{t+1})- \nabla F(x_t) = \nabla_x \widetilde{\mathcal L}_{\lambda_t}(x_t, \mathbf y_t, \mathbf z_{t+1}) - \nabla \mathcal L^*_{\lambda_t} (x_t) + \nabla \mathcal L^*_{\lambda_t} (x_t) - \nabla F(x_t)
    \end{equation}
    where
    \begin{equation}\label{eq:q}
       \nabla_x \widetilde{\mathcal L}_{\lambda_t}(x_t, \mathbf y_t, \mathbf z_{t+1}) = \nabla_x f(x_t, \mathbf y_{t+1}) + \lambda_t \sum_{i = 1}^k \nabla_x g_i(x_t, y_{i, t+1}, z_{-i, t+1}) - \nabla_x g_i (x_t, \mathbf z_{t+1})
    \end{equation}
    and 
    \begin{equation}\label{eq:l}
        \nabla \mathcal L^*_{\lambda_t} (x_t) = \nabla_x \mathcal{L}_{\lambda_t} (x_t, \mathbf y^*_{\lambda}(x_t)) = \nabla_x f(x_t, \mathbf y^*_{\lambda_t}(x_t)) + \lambda_t \sum_{i = 1}^k \nabla_x g_i(x_t, y_{i, \lambda_t}^*(x_t), y^*_{-i}(x_t)) - \nabla_x g_i (x_t, \mathbf y^*(x_t)).
    \end{equation}
    Substituting Equations (\ref{eq:q}) and (\ref{eq:l}) into (\ref{eq:q-f}), we get:
    \begin{align*}
         \nabla_x \widetilde{\mathcal L}_{\lambda_t}(x_t, \mathbf y_t, \mathbf z_{t+1})- \nabla F(x_t)
         &=\nabla_x f(x_t, \mathbf y_{t+1})- \nabla_x f(x_t, \mathbf y^*_{\lambda_t}(x_t)) \\
         &\quad +  \lambda_t \sum_{i = 1}^k \bigg(\nabla_x g_i(x_t, y_{i, t+1}, z_{-i, t+1}) - \nabla_x g_i(x_t, y_{i, \lambda_t}^*(x_t), y^*_{-i}(x_t))\bigg) \\
         &\quad+ \lambda_t\sum_{i = 1}^k \bigg(\nabla_x g_i (x_t, \mathbf y^*(x_t)) - \nabla_x g_i (x_t, \mathbf z_{t+1})\bigg) + \nabla \mathcal L^*_{\lambda_t} (x_t) - \nabla F(x_t).
    \end{align*}
    Next, we take the norm and using the fact that $f$ and $g$ are both smooth, we obtain:
    \begin{align*}
        \norm{\nabla_x \widetilde{\mathcal L}_{\lambda_t}(x_t, \mathbf y_t, \mathbf z_{t+1})- \nabla F(x_t)} 
        \le 
        \ell_{f,1} \norm{\mathbf y_{t+1} -  \mathbf y^*_{\lambda_t}(x_t)} 
        &+ \lambda_t \sum_{i = 1}^k \norm{( y_{i, t+1}, z_{-i, t+1}) 
        - (y_{i, \lambda_t}^*(x_t), y^*_{-i}(x_t)) }\\
        &+ \lambda_t \sum_{i = 1}^k \norm{\mathbf y^*(x_t) - \mathbf z_{t+1}} + \norm{\nabla \mathcal L^*_{\lambda_t} (x_t) - \nabla F(x_t)}
    \end{align*}
    where the last term $\norm{\nabla \mathcal L^*_{\lambda_t} (x_t) - \nabla F(x_t)}\le kC_{\lambda}/\lambda$ is exactly Theorem \ref{bound2}.
    Note that $$\norm{( y_{i, t+1}, z_{-i, t+1}) - (y_{i, \lambda_t}^*(x_t), y^*_{-i}(x_t)) } \le \sqrt {\norm{\mathbf y_{t+1} - y^*_\lambda(x_t) }^2+ \norm{\mathbf z_{t+1} - \mathbf y^*(x_t)}^2}.$$ Using this and the fact that $(a+b+c+d)^2 \le 4(a^2 + b^2 + c^2 + d^2)$, we obtain:
    \begin{align*}
         \norm{\nabla_x \widetilde{\mathcal L}_{\lambda_t}(x_t, \mathbf y_t, \mathbf z_{t+1})- \nabla F(x_t)}^2 \le 4\bigg( (\ell_{f,1}^2 + k^2 \lambda_t^2) \norm{\mathbf y_{t+1} - \mathbf y^*_\lambda(x_t)}^2 + 2 k^2 \lambda_t^2\norm{\mathbf z_{t+1} - \mathbf y^*(x_t)}^2 + \frac{k^2C_{\lambda_t}^2}{\lambda_t^2}\bigg).
    \end{align*} 
    Finally, substituting the previous result into (\ref{F-F}) yields the desired bound.

\end{proof}

\begin{lemma}\label{lagrangian_convex}
    Choose $\lambda_i = \lambda$ for all $i\in [k]$.  If $\lambda\ge \frac{2\ell_{f,1}}{\mu_g}$, then $\mathcal{L}_\lambda(x, \mathbf{y})$ is $\bigg(\frac{\mu_g\lambda}{2}\bigg)$-strongly convex in $\mathbf y$.
\end{lemma}
\begin{proof}
    We use $I_{n_i}$ to denote the $n_i\times n_i$ identity matrix. In particular, \[I_N = \begin{pmatrix}
        I_{n_1} & & 0\\
         & \ddots & \\
         0 & & I_{n_k}
    \end{pmatrix}, \quad N= \sum_{i = 1}^k n_i.\]
    Recall 
    \[\mathcal L_{\lambda}(x, \mathbf y) = f(x,\mathbf{y})+ \sum_{i=1}^{k}\lambda_i(g_i(x,y_i, {y}_{-i}^*(x))-g(x, \mathbf{y}^*(x))).\]
    Hence 
    \[\nabla^2_{yy} \mathcal L_{\lambda}(x, \mathbf y) = \nabla^2_{yy}f(x,\mathbf y)+ \sum_{i = 1}^k \lambda_i \nabla^2_{yy} g_i (x, \mathbf y).\]
    where 
    \[ \nabla^2_{yy}f(x,\mathbf y) = \bigg[ \nabla^2_{y_iy_j} f(x,\mathbf y)\bigg]_{i,j=1}^k\]
    \[\nabla^2_{yy}g_i(x,\mathbf y) = \diag \bigg[\nabla^2_{y_iy_i}g_i(x, \mathbf y)\bigg]_{i = 1}^k. \]
    are both $N \times N$ matrices. $\nabla^2_{yy}g_i(x,\mathbf y)$ is a block diagonal matrix because all entries are fixed to be $y^*_j(x)$ for all $j \neq i$. Therefore, $\nabla^2_{y_j y_{j'}} g_i(x,\mathbf{y}) = 0 ~\forall j \neq j'$. We now use \textit{Fact} \ref{additive}.

    \begin{fact}\label{additive}
        If $A,B$ are symmetric and \[A \succeq \alpha I, B\succeq \beta I,\] then \[A+B \succeq (\alpha + \beta)I.\]
   \end{fact}

    By Assumption \ref{ass:mug}, $g_i(x,\mathbf y)$ is $\mu_g$-strongly convex in $y_i$. Hence,
    \[
    \nabla^2_{y_iy_i}g_i(x,\mathbf y) 
    \succeq \mu_g I_{n_i} \implies \nabla^2_{yy}g_i(x,\mathbf y) 
    \succeq \mu_g I_N
.    \] 
    Thus, picking $\lambda_i = \lambda$ for all $i$ we get
    \[\sum_{i = 1}^k \lambda_i \nabla^2_{yy}g_i(x,\mathbf y)
    \succeq (\min_i \lambda_i) \mu_gI_N = \ \mu_g \lambda  I_{N}.\]

    On the other hand, we also assumed $f(x,\mathbf y)$ is $\ell_{f,1}$-smooth in $\mathbf y$. So 
    \[-\ell_{f,1}I_{N} \preceq \nabla^2_{yy} f(x, \mathbf y) \preceq \ell_{f,1} I_{N}\] 
    Applying \textit{Fact} \ref{additive} again to $\nabla^2_{yy} \mathcal L_{\lambda}(x, \mathbf y)$, we get
    \[\nabla^2_{yy} \mathcal L_{\lambda}(x, \mathbf y) \succeq ( -\ell_{f,1} + \mu_g \lambda) I_N\]
    Imposing the condition $\lambda\ge \frac{2\ell_{f,1}}{\mu_g}$ gives 
    \[\nabla^2_{yy} \mathcal L_{\lambda}(x, \mathbf y))\succeq \frac{\mu_g\lambda}{2}I_N. \]
    This proves the statement.
\end{proof}

\begin{lemma}\label{dy*t}
    Choose $\lambda_{1,i} = \lambda_1$ and $\lambda_{2,i} = \lambda_2$ for all $i\in [k]$, then for any $x_1,x_2\in X$ and for any $k\lambda_2\ge k\lambda_1 \ge \frac{\ell_{f,1}}{\mu_g}$, we have 
    \begin{align*}
        \lVert y_{i,\lambda_1}^*(x_1) -  y_{i,\lambda_2}^*(x_2)\rVert &\le \bigg(\lVert x_1-x_2\rVert(\ell_{f,1} + \ell_{g,1}\lambda_{2,i}) +  (\lambda_{2,i}- \lambda_{1,i}) \frac{\ell_{f,0}}{\lambda_{1,i}}\bigg) \frac{2}{\mu_g\lambda_2}.
    \end{align*}
\end{lemma}
\begin{proof}
By the optimality condition of $ \mathcal{L}_{\mathbf{\lambda}} (x_1,\mathbf{y}_{\lambda_1}^*(x_1))$ at $y^*_{i,\lambda_1}(x_1)$ with input $x_1$ and $\lambda_1$, we have
    \begin{align*}
        \nabla_{y_i} \mathcal{L}_{\mathbf{\lambda}_1} (x_1,\mathbf{y}_{\lambda_1}^*(x_1))&= \nabla_{y_i}f(x_1,\mathbf{y}_{\lambda_1}^*(x_1)) + \lambda_{1,i} \nabla y_i g_i(x_1,\mathbf{y}_{\lambda_1}^*(x_1)) = 0 \\&\implies \lVert \nabla_{y_i} g_i(x_1,\mathbf{y}_{\lambda_1}^*(x_1))\rVert \le \frac{\ell_{f,0}}{\lambda_{1}}.
    \end{align*}
Consider the following
    \begin{align*}
&\nabla_{y_i} \mathcal{L}_{\mathbf{\lambda}_2} (x_2,\mathbf{y}_{\lambda_1}^*(x_1))\\
&=\nabla_{y_i}f(x_2,\mathbf{y}_{\lambda_1}^*(x_1))+ \lambda_{2} \nabla_{ y_i} g_i(x_2,\mathbf{y}_{\lambda_1}^*(x_1))\\
&= \bigg(\nabla_{y_i}f(x_2,\mathbf{y}_{\lambda_1}^*(x_1)) - \nabla_{y_i}f(x_1,\mathbf{y}_{\lambda_1}^*(x_1))\bigg) + \nabla_{y_i}f(x_1,\mathbf{y}_{\lambda_1}^*(x_1)) \\
&\quad \quad+ \lambda_{2}\bigg(\nabla_{y_i} g_i(x_2,\mathbf{y}_{\lambda_1}^*(x_1))- \nabla_{y_i} g_i(x_1,\mathbf{y}_{\lambda_1}^*(x_1))\bigg)  + \lambda_{2}g_i(x_1,\mathbf{y}_{\lambda_1}^*(x_1))\\
&= \bigg(\nabla_{y_i}f(x_2,\mathbf{y}_{\lambda_1}^*(x_1)) - \nabla_{y_i}f(x_1,\mathbf{y}_{\lambda_1}^*(x_1))\bigg) \\
&\quad \quad+\lambda_{2}\bigg(\nabla_{y_i} g_i(x_2,\mathbf{y}_{\lambda_1}^*(x_1))-\nabla_{y_i} g_i(x_1,\mathbf{y}_{\lambda_1}^*(x_1))\bigg) +(\lambda_{2}- \lambda_{1})\nabla_{y_i} g_i(x_1,\mathbf{y}_{\lambda_1}^*(x_1)).
\end{align*}
Now, using the smoothness condition of $f$ and $g$ in $x$, we get 
\begin{align*}
    &\lVert \nabla_{y_i}f(x_2,\mathbf{y}_{\lambda_1}^*(x_1)) + \lambda_{2} \nabla_{ y_i} g_i(x_2,\mathbf{y}_{\lambda_1}^*(x_1))\rVert\\ &\le \ell_{f,1}\lVert x_1-x_2\rVert +\ell_{g,1}\lambda_{2,i}\lVert x_1-x_2\rVert + (\lambda_{2}- \lambda_{1}) \frac{\ell_{f,0}}{\lambda_{1,i}}.
\end{align*}
By $\bigg(\frac{\mu_g \lambda}{2}\bigg)$-strong-convexity of $L_{\lambda_2}(x_2,\mathbf{y})$ in $y_i$, we get
\begin{align*}
    &\lVert y_{i,\lambda_1}^*(x_1) -  y_{i,\lambda_2}^*(x_2)\rVert\\ &\le \frac{2}{\mu_g\lambda_2} \lVert \nabla_{y_i} \mathcal{L}_{\mathbf{\lambda}_2} (x_2,\mathbf{y}_{\lambda_1}^*(x_1))\rVert \\
    &\le \bigg(\lVert x_1-x_2\rVert(\ell_{f,1} + \ell_{g,1}\lambda_{2,i}) +  (\lambda_{2,i}- \lambda_{1,i}) \frac{\ell_{f,0}}{\lambda_{1}}\bigg) \frac{2}{\mu_g\lambda_2}.
\end{align*}
\end{proof}

\subsection{Auxiliary Lemmas for Section 6}
Lemma \ref{yzb} is an auxiliary lemma that bounds the discrepancy between the approximated Lagrangian minimizer in line \ref{algo:lagragianapp} and the monotone game equilibrium with some error. 
\begin{lemma}\label{yzb}
    $\lVert \mathbf y_{t+1}- \mathbf z_{t+1} \rVert \le \lVert \mathbf y_{t+1}- \mathbf y^*(x)\rVert +\frac{C_z}{\mu_g \sqrt{M_{z,t}}}$.
\end{lemma}
\begin{proof}
    By the triangle inequality, we obtain:
    \[\lVert \mathbf y_{t+1}- \mathbf z_{t+1} \rVert \le \lVert \mathbf y_{t+1}- \mathbf y^*(x) \rVert + \lVert \mathbf y^*- \mathbf z_{t+1} \rVert\]
    Then, apply the fact \[\lVert V_{z} (x,\mathbf z_{M_{z,t}})\rVert \le \frac{C_z}{\sqrt{M_{z,t}}} \]  to the second term in the sum.
\end{proof}

\begin{lemma}\label{dx}
    $\lVert x_t- x_{t-1} \rVert \le \eta_{t-1} (\ell_{f,0} + 2k\ell_{g,0})$.    
\end{lemma}
\begin{proof}
By the updating rule at line \ref{algo:x}, we have \[x_t- x_{t-1}  \le \eta_{t-1}\nabla_x\mathcal{L}_{\mathbf{\lambda}_t} (x_t,\mathbf{y}_{t+1})\]
and 
    \begin{align*}
        \nabla_x\mathcal{L}_{\mathbf{\lambda}_t} (x_t,\mathbf{y}_{t+1}) = & \nabla_xf(x_t,\mathbf{y}_{t}) + \lambda_t\sum_{i=1}^{k} \nabla_x g_i(x_t,y_{i,t+1}, z_{-i,t+1})-\lambda_t\sum_{i=1}^k  \nabla_xg_i(x_t,\mathbf{z}_{t+1}).
    \end{align*}
By Assumption \ref{ass:flipx} we have 
\begin{align*}
    \lVert \nabla_x\mathcal{L}_{\mathbf{\lambda}_t} (x_t,\mathbf{y}_{t+1}) \rVert \le \ell_{f,0} + 2k\ell_{g,0}.
\end{align*}
Putting everything together yields the result. 
\end{proof}

\end{document}